\documentclass{amsart}
\usepackage[utf8]{inputenc}
\usepackage[hidelinks]{hyperref}
\usepackage[noadjust]{cite}
\usepackage{aliascnt, amsthm, amsmath, amssymb, enumerate}

\theoremstyle{definition}
\newtheorem{corollary}{Corollary}[section]

\newcommand{\mynewtheorem}[2]{
  \newaliascnt{#1}{corollary}
  \newtheorem{#1}[#1]{#2}
  \aliascntresetthe{#1}
  \expandafter\def\csname #1autorefname\endcsname{#2}
}

\mynewtheorem{definition}{Definition}
\mynewtheorem{example}{Example}
\mynewtheorem{lemma}{Lemma}
\mynewtheorem{proposition}{Proposition}
\mynewtheorem{remark}{Remark}
\mynewtheorem{theorem}{Theorem}

\DeclareMathOperator{\characteristic}{char}

\begin{document}

\title{Hom-associative Ore extensions and weak unitalizations}
\author{Per B\"ack, Johan Richter, \and Sergei Silvestrov}
\address{Division of Applied Mathematics, The School of Education, Culture and Communication, M\"alar\-dalen  University,  Box  883,  SE-721  23  V\"aster\r{a}s, Sweden}
\email{per.back@mdh.se, johan.richter@mdh.se, \and sergei.silvestrov@mdh.se}
\def\emailaddrname{{\itshape E-mail addresses}}

\subjclass[2010]{17A30, 17A01}

\begin{abstract}
We introduce hom-associative Ore extensions as non-unital, non-associative Ore extensions with a hom-associative multiplication, and give some necessary and sufficient conditions when such exist. Within this framework, we construct families of hom-associative quantum planes, universal enveloping algebras of a Lie algebra, and Weyl algebras, all being hom-associative generalizations of their classical counterparts, as well as prove that the latter are simple. We also provide a way of embedding any multiplicative hom-associative algebra into a multiplicative, weakly unital hom-associative algebra, which we call a weak unitalization.
\end{abstract}

\maketitle
\section{Introduction}
Hom-Lie algebras and related hom-algebra structures have recently become a subject of growing interest and extensive investigations, in part due to the prospect of providing a general framework in which one can produce many types of natural deformations of (Lie) algebras, in particular $q$-deformations which are of interest both in mathematics and in physics. One of the main initial motivations for this development came from mathematical physics works on $q$-deformations of infinite-dimensional
algebras, primarily the $q$-deformed Heisenberg algebras ($q$-deformed Weyl algebras), oscillator algebras, and the Virasoro algebra
\cite{AizawaSato,ChaiElinPop,ChaiKuLukPopPresn,ChaiKuLuk,ChaiPopPres,CurtrZachos1,CurtrFairlZachos,DaskaloyannisGendefVir,Hu,K92,LiuKQuantumCentExt,LiuKQCharQuantWittAlg,LiuKQPhDthesis}.

Quasi-Lie algebras, subclasses of quasi-hom-Lie algebras, and hom-Lie algebras as well as their general colored (graded) counterparts were introduced between 2003 and 2005 in \cite{HLS,LS1,LS2,LSGradedquasiLiealg,Czech:witt}. Further on, between 2006 and 2008, Makhlouf and Silvestrov introduced the notions of hom-associative algebras, hom-(co, bi)algebras and hom-Hopf algebras, and also studied their properties \cite{MS1,MS2,MS3}. A hom-associative algebra, being a generalization of an associative algebra with the associativity axiom extended by a linear twisting map, is always hom-Lie admissible, meaning that the commutator multiplication in any hom-associative algebra yields a hom-Lie algebra \cite{MS1}. Whereas associativity is replaced by hom-associativity in hom-associative algebras, hom-coassociativity for hom-coalgebras can be considered in a similar way.

One of the main tools in these important developments and in many constructions of examples and classes of hom-algebra structures in physics and in mathematics are based on twisted derivations, or $\sigma$-derivations, which are generalized derivations twisting the Leibniz rule by means of a linear map. These types of twisted derivation maps are central for the associative Ore extension algebras, or rings, introduced in algebra in the 1930s, generalizing crossed product (semidirect product) algebras, or rings, incorporating both actions and twisted derivations.

Non-associative Ore extensions on the other hand were first introduced in 2015 and in the unital case, by Nystedt, {\"O}inert, and Richter \cite{2015arXiv150901436N} (see also \cite{2017arXiv1705.02778} for an extension to monoid Ore extensions). In the present article, we generalize this construction to the non-unital case, as well as investigate when these non-unital, non-associative Ore extensions are hom-associative. Finding necessary and sufficient conditions for such to exist, we are also able to construct families of hom-associative quantum planes (\autoref{ex:hom-quant}), universal enveloping algebras of a Lie algebra (\autoref{ex:hom-env}), and Weyl algebras (\autoref{ex:hom-weyl}), all being hom-associative generalizations of their classical counterparts. We do not make use of any previous results about non-associative Ore extensions, but our construction of hom-associative Weyl algebras has some similarities to the non-associative Weyl algebras in \cite{2015arXiv150901436N}; for instance they both are simple. At last, in \autoref{sec:weak-unitalization}, we prove constructively that any multiplicative hom-associative algebra can be embedded in a multiplicative, weakly unital hom-associative algebra.

\section{Preliminaries}
In this section, we present some definitions and review some results from the theory of hom-associative algebras and that of non-associative Ore extensions.
\subsection{Hom-associative algebras}
Here we define what we mean for an algebraic structure to be \emph{hom-associative}, and review a couple of results concerning the construction of them. First, throughout this paper, by \emph{non-associative} algebras we mean algebras which are not necessarily associative, which includes in particular associative algebras by definition. We also follow the convention of calling a non-associative algebra $A$ \emph{unital} if there exist an element $1\in A$ such that for any element $a\in A$, $a\cdot 1=1\cdot a=a$. By \emph{non-unital} algebras, we mean algebras which are not necessarily unital, which includes also unital algebras as a subclass.

\begin{definition}[Hom-associative algebra]\label{def:hom-assoc-algebra} A \emph{hom-associative algebra} over an associative, commutative, and unital ring $R$, is a triple $(M,\cdot,\alpha)$ consisting of an $R$-module $M$, a binary operation $\cdot\colon M\times M\to M$ linear over $R$ in both arguments, and an $R$-linear map $\alpha\colon M\to M$ satisfying, for all $a,b,c\in M$,
\begin{equation}
\alpha(a)\cdot(b\cdot c)=(a\cdot b)\cdot\alpha(c). \label{eq:hom-condition}
\end{equation}
\end{definition}
Since $\alpha$ twists the associativity, we will refer to it as the \emph{twisting map}, and unless otherwise stated, it is understood that $\alpha$ without any further reference will always denote the twisting map of a hom-associative algebra.
\begin{remark}
A hom-associative algebra over $R$ is in particular a non-unital, non-associative $R$-algebra, and in case $\alpha$ is the identity map, a non-unital, associative $R$-algebra.
\end{remark}

Furthermore, if the twisting map $\alpha$ is also multiplicative, i.e. if $\alpha(a\cdot b) = \alpha(a)\cdot \alpha(b)$ for all elements $a$ and $b$ in the algebra,
then we say that the hom-associative algebra is \emph{multiplicative}.

\begin{definition}[Morphism of hom-associative algebras]\label{def:morphism} A \emph{morphism} between two hom-associative algebras $A$ and $A'$ with twisting maps $\alpha$ and $\alpha'$ respectively, is an algebra homomorphism $f\colon A\to A'$ such that $f\circ \alpha= \alpha'\circ f$. If $f$ is also bijective, the two are \emph{isomorphic}, written $A\cong A'$.
\end{definition}

\begin{definition}[Hom-associative subalgebra] Let $A$ be a hom-associative algebra with twisting map $\alpha$. A \emph{hom-associative subalgebra} $B$ of $A$ is a subalgebra of $A$ that is also a hom-associative algebra with twisting map given by the restriction of $\alpha$ to $B$.
\end{definition}

\begin{definition}[Hom-ideal] A \emph{hom-ideal} of a hom-associative algebra is an algebra ideal $I$ such that $\alpha(I)\subseteq I$.
\end{definition}

In the classical setting, an ideal is in particular a subalgebra. With the above definition, the analogue is also true for a hom-associative algebra, in that a hom-ideal is a hom-associative subalgebra.

\begin{definition}[Hom-simplicity] We say that a hom-associative algebra $A$ is \emph{hom-simple} provided its only hom-ideals are $0$ and $A$.
\end{definition}

In particular, we see that any simple hom-associative algebra is also hom-simple, while the converse need not be true; there may exist ideals that are not invariant under $\alpha$.

\begin{definition}[Hom-associative ring]\label{def:hom-ring} A \emph{hom-associative ring} can be seen as a hom-associative algebra over the ring of integers.
\end{definition}

\begin{definition}[Weakly unital hom-associative algebra]\label{def:weak-hom}
Let $A$ be a hom-associ\-ative algebra. If for all $a\in A$, $e\cdot a=a\cdot e=\alpha(a)$ for some $e\in A$, we say that $A$ is \emph{weakly unital} with \emph{weak unit} $e$.
\end{definition}

\begin{remark}Any unital, hom-associative algebra with twisting map $\alpha$ is weakly unital with weak unit $\alpha(1)$, since by hom-associativity
\begin{equation*}
\alpha(1)\cdot a=\alpha(1)\cdot (1\cdot a)=(1\cdot 1)\cdot\alpha(a)=\alpha(a)=\alpha(a)\cdot(1\cdot 1)=(a\cdot 1)\cdot \alpha(1)=a\cdot\alpha(1).
\end{equation*}
\end{remark}

Any non-unital, associative algebra can be extended to a non-trivial hom-associ\-ative algebra, which the following proposition demonstrates:

\begin{proposition}[\cite{lietheoryyau}]\label{prop:star-alpha-mult} Let $A$ be a non-unital, associative algebra, $\alpha$ an algebra endomorphism on $A$ and define $*\colon A\times A\to A$ by $a* b:=\alpha(a\cdot b)$ for all $a,b\in A$. Then $(A,*,\alpha)$ is a hom-associative algebra.
\end{proposition}

\begin{proof}
Linearity follows immediately, while for all $a,b,c\in A$, we have
\begin{align*}
\alpha(a)* (b* c)&=\alpha(a)* (\alpha(b\cdot c))=\alpha(\alpha(a)\cdot \alpha(b\cdot c))=\alpha(\alpha(a\cdot b\cdot c)),\\
(a* b)* \alpha(c)&=\alpha(a\cdot b)* \alpha(c)=\alpha(\alpha(a\cdot b)\cdot\alpha(c))=\alpha(\alpha(a\cdot b\cdot c)),
\end{align*}
which proves that $(A,* ,\alpha)$ is hom-associative.
\end{proof}
Note that we are abusing the notation in \autoref{def:hom-assoc-algebra} a bit here; $A$ in $(A,*,\alpha)$ does really denote the algebra and not only its module structure. From now on, we will always refer to this construction when writing $*$.

\begin{corollary}[\cite{2009arXiv0904.4874F}]\label{cor:weak-unit} If $A$ is a unital, associative algebra, then $(A,*,\alpha)$ is weakly unital with weak unit 1.
\end{corollary}

\begin{proof}
$1* x=\alpha(1\cdot x)=\alpha(x)=\alpha(x\cdot1)=x* 1.$
\end{proof}

\subsection{Non-unital, non-associative Ore extensions}
In this section, we define non-unital, non-associative Ore extensions, together with some new terminology.
\begin{definition}[Left $R$-additivity] If $R$ is a non-unital, non-associative ring, we say that a map $\beta\colon R\to R$ is \emph{left $R$-additive} if for all $r,s,t\in R$, $r\cdot\beta(s+t)=r\cdot\left(\beta(s)+\beta(t)\right)$.
\end{definition}
In what follows, $\mathbb{N}$ will always denote the set of non-negative integers, and $\mathbb{N}_{>0}$ the set of positive integers. Now, given a non-unital, non-associative ring $R$ with left $R$-additive maps $\delta\colon R\to R$ and $\sigma\colon R\to R$, by a \emph{non-unital, non-associative Ore extension} of $R$, $R[X;\sigma,\delta]$, we mean the set of formal sums $\sum_{i\in\mathbb{N}} a_i X^i,\ a_i\in R,$ called polynomials, with finitely many $a_i$ nonzero, endowed with the addition
\begin{equation*}
\sum_{i\in\mathbb{N}}a_iX^i + \sum_{i\in\mathbb{N}}b_i X^i = \sum_{i\in\mathbb{N}}(a_i+b_i)X^i,\quad a_i,b_i\in R,
\end{equation*}
where two polynomials are equal if and only if their corresponding coefficients are equal, and for all $a,b\in R$ and $m,n\in\mathbb{N}$, a multiplication
\begin{equation}
aX^m\cdot bX^n=\sum_{i\in \mathbb{N}}\left(a\cdot\pi^m_i(b)\right)X^{i+n}.\label{eq:ore-mult}
\end{equation}
Here $\pi_i^m$ denotes the sum of all $\binom{m}{i}$ possible compositions of $i$ copies of $\sigma$ and $m-i$ copies of $\delta$ in arbitrary order. Then, for example $\pi_0^0=\mathrm{id}_R$ and $\pi^3_1=\sigma\circ\delta\circ\delta+\delta\circ\sigma\circ\delta+\delta\circ\delta\circ\sigma$. We also extend the definition of $\pi_i^m$ by setting $\pi^m_i\equiv0$ whenever $i<0$, or $i>m$. Imposing distributivity of the multiplication over addition makes $R[X;\sigma,\delta]$ a ring. In the special case when $\sigma=\mathrm{id}_R$, we say that $R[X;\mathrm{id}_R,\delta]$ is a \emph{non-unital, non-associative differential polynomial ring}, and when $\delta\equiv0$, $R[X;\sigma,0]$ is said to be a \emph{non-unital, non-associative skew polynomial ring}.

Note that when $m=n=0$, $aX^0\cdot bX^0 = \sum_{i\in\mathbb{N}}\left(a\cdot\pi_i^0(b)\right)X^i=(a\cdot b)X^0,$ so $R\cong RX^0$ by the isomorphism $r\mapsto rX^0$ for any $r\in R$. Since $RX^0$ is a subring of  $R[X;\sigma,\delta]$, we can view $R$ as a subring of $R[X;\sigma,\delta]$, making sense of expressions like $a\cdot bX^0$.

\begin{remark}If $R$ contains a unit, we write $X$ for the formal sum $\sum_{i\in\mathbb{N}}a_i X^i$ with $a_1=1$ and $a_i=0$ when $i\neq 1$. It does not necessarily make sense to think of $X$ as an element of the non-associative Ore extension if $R$ is not unital.
\end{remark}
The left-distributivity of the multiplication over addition forces $\delta$ and $\sigma$ to be left $R$-additive: for any $r,s,t\in R$, $rX\cdot(s+t)=rX\cdot s+rX\cdot t$, and by expanding the left- and right-hand side,
\begin{align*}
rX\cdot(s+t)&=r\cdot\sigma(s+t)X+r\cdot\delta(s+t),\\
rX\cdot s+rX\cdot t&=r\cdot\sigma(s)X+r\cdot\delta(s)+r\cdot\sigma(t)X+r\cdot\delta(t),
\end{align*}
so by comparing coefficients, we arrive at the desired conclusion.

\begin{definition}[$\sigma$-derivation]\label{def:sigma-derivation} Let $R$ be a non-unital, non-associative ring where $\sigma$ is an endomorphism and $\delta$ an additive map on $R$. Then $\delta$ is called a \emph{$\sigma$-derivation} if $\delta(a\cdot b)=\sigma(a)\cdot\delta(b)+\delta(a)\cdot b$ holds for all $a,b\in R$. If $\sigma=\mathrm{id}_R$, $\delta$ is a $\emph{derivation}$.
\end{definition}

\begin{remark}\label{re:sigma-derivation}
If $R$ and $\sigma$ are unital and $\delta$ a $\sigma$-derivation, then $\delta(1)=\delta(1\cdot 1)=2\cdot\delta(1)$, so that $\delta(1)=0$. Furthermore, if $R$ is also associative, then it is both a necessary and sufficient condition that $\sigma$ be an endomorphism and $\delta$ a $\sigma$-derivation on $R$ for the unital, associative Ore extension $R[X;\sigma,\delta]$ to exist.
\end{remark}

\begin{definition}[Homogeneous map] Let $R[X;\sigma,\delta]$ be a non-unital non-associ\-ative Ore extension of a non-unital, non-associative ring $R$. Then we say that a map $\beta\colon R[X;\sigma,\delta]\to R[X;\sigma,\delta]$ is \emph{homogeneous} if for all $a\in R$ and $m\in\mathbb{N}$, $\beta(aX^m)=\beta(a)X^m$. If $\gamma\colon R\to R$ is any (additive) map, we may \emph{extend it homogeneously} to $R[X;\sigma,\delta]$ by defining $\gamma(aX^m):=\gamma(a)X^m$ (imposing additivity).
\end{definition}

\section{Non-associative Ore extensions of non-associative rings}
We use this small section to present a couple of results that hold true for any non-unital, non-associative Ore extension of a non-unital, non-associative ring.

\begin{lemma}[Homogeneously extended ring endomorphism]\label{lem:endo-extension}
Let $R[X;\sigma,\delta]$ be a non-unital, non-associative Ore extension of a non-unital, non-associative ring $R$. If $\gamma$ is an endomorphism on $R$, then the homogeneously extended map is an endomorphism on $R[X;\sigma,\delta]$ if and only if
\begin{equation}
\gamma(a)\cdot\pi_i^m(\gamma(b))=\gamma(a)\cdot\gamma(\pi_i^m(b)),\quad\text{for all } i,m\in\mathbb{N}\text{ and } a,b\in R.\label{eq:homogeneous-endomorphism}
\end{equation}
\end{lemma}
\begin{proof}
Additivity follows from the definition, while for any monomials $aX^m$ and $bX^n$,
\begin{align*}
\gamma(aX^m\cdot bX^n)&=\gamma\left(\sum_{i\in\mathbb{N}}a\cdot\pi_i^m(b)X^{i+n}\right)=\sum_{i\in\mathbb{N}}\gamma(a)\cdot\gamma\big(\pi_i^m(b)\big)X^{i+n},\\
\gamma(aX^m)\cdot\gamma(bX^n)&=\gamma(a)X^m\cdot\gamma(b)X^n=\sum_{i\in\mathbb{N}}\gamma(a)\cdot\pi_i^m(\gamma(b))X^{i+n}.
\end{align*}
Comparing coefficients between the two completes the proof.
\end{proof}

\begin{corollary}[Homogeneously extended unital ring endomorphism]\label{cor:unital-endo-extension} Let $R[X;\sigma,\delta]$ be a unital, non-associative Ore extension of a unital, non-associative ring $R$. If $\alpha$ is an endomorphism on $R$ and there exists an $a\in R$ such that $\alpha(a)=1$, then the homogeneously extended map on $R[X;\sigma,\delta]$ is an endomorphism if and only if $\alpha$ commutes with $\delta$ and $\sigma$.
\end{corollary}
\begin{proof}
This follows from \autoref{lem:endo-extension} by choosing $a$ so that $\alpha(a)=1$: if $\alpha$ commutes with $\delta$ and $\sigma$, then $\pi_i^m(\alpha(b))=\alpha(\pi_i^m(b))$. On the other hand, if $\pi_i^m(\alpha(b))=\alpha(\pi_i^m(b))$, then by choosing $m=1$ and $i=0$, $\pi_0^1(\alpha(b))=\delta(\alpha(b)),$ and $\alpha(\pi_0^1(b))=\alpha(\delta(b)).$ If choosing $m=1$ and $i=1$,
$\pi_1^1(\alpha(b))=\sigma(\alpha(b)),$ $\alpha(\pi_1^1(b))=\alpha(\sigma(b)).$\qedhere
\end{proof}

\section{Hom-associative Ore extensions of non-associative rings}
The following section is devoted to the question what non-unital, non-associative Ore extensions of non-unital, non-associative rings $R$ are hom-associative?

\begin{proposition}[Hom-associative Ore extension]\label{prop:hom-ass-ore-condition}
Let $R[X;\sigma,\delta]$ be a non-unital, non-associative Ore extension of a non-unital, non-associative ring $R$. Furthermore, let $\alpha_{i,j}(a)\in R$ be dependent on $a\in R$ and $i,j\in\mathbb{N}_{>0}$, and put for an additive map $\alpha\colon R[X;\sigma,\delta]\to R[X;\sigma,\delta]$,
\begin{equation}
\alpha\left(aX^m\right)=\sum_{i\in\mathbb{N}} \alpha_{i+1,m+1}(a)X^i, \quad\forall a\in R,\forall m\in\mathbb{N}.\label{eq:alpha-definition}
\end{equation}
Then $R[X;\sigma,\delta]$ is hom-associative with the twisting map $\alpha$ if and only if for all $a,b,c\in R$ and $k,m,n,p\in\mathbb{N}$,
\begin{equation}
\sum_{j\in\mathbb{N}}\sum_{i\in\mathbb{N}}\alpha_{i+1,m+1}(a)\cdot\pi_{k-j}^i\left(b\cdot\pi_{j-p}^n(c)\right) = \sum_{j\in\mathbb{N}}\sum_{i\in\mathbb{N}}\left(a\cdot\pi_i^m(b)\right)\cdot \pi_{k-j}^{i+n}\left(\alpha_{j+1,p+1}(c)\right).\label{eq:hom-ore-general}
\end{equation}
\end{proposition}

\begin{proof}
For any $a,b,c\in R$ and $m,n,p\in\mathbb{N}$,
{\allowdisplaybreaks
\begin{align*}
\alpha\left(aX^m\right)\cdot \left(b X^n \cdot cX^p\right)&=\alpha\left(aX^m\right)\cdot\left(\sum_{q\in\mathbb{N}}\left(b\cdot\pi_q^n(c)\right)X^{q+p}\right)\\
&=\sum_{q\in\mathbb{N}}\alpha\left(aX^m\right)\cdot\left(\left(b\cdot\pi_q^n(c)\right)X^{q+p}\right)\\
&=\sum_{q\in\mathbb{N}}\sum_{i\in\mathbb{N}}\alpha_{i+1,m+1}(a)X^i\cdot\left(\left(b\cdot\pi_q^n(c)\right)X^{q+p}\right)\\
&=\sum_{q\in\mathbb{N}}\sum_{i\in\mathbb{N}}\sum_{l\in\mathbb{N}}\alpha_{i+1,m+1}(a)\cdot\pi_l^i\left(b\cdot\pi_q^n(c)\right)X^{l+q+p}\\
&=\sum_{l\in\mathbb{N}}\sum_{j\in\mathbb{N}}\sum_{i\in\mathbb{N}}\alpha_{i+1,m+1}(a)\cdot\pi_l^i\left(b\cdot\pi_{j-p}^n(c)\right)X^{l+j}\\
&=\sum_{k\in\mathbb{N}}\sum_{j\in\mathbb{N}}\sum_{i\in\mathbb{N}}\alpha_{i+1,m+1}(a)\cdot\pi_{k-j}^i\left(b\cdot\pi_{j-p}^n(c)\right)X^{k},\\
\left(aX^m\cdot bX^n\right)\cdot \alpha\left(cX^p\right)&=\left(\sum_{i\in\mathbb{N}}\left(a\cdot\pi_i^m(b)\right)X^{i+n}\right)\cdot \alpha\left(cX^p\right)\\
&=\sum_{i\in\mathbb{N}}\left(a\cdot\pi_i^m(b)\right)X^{i+n}\cdot \alpha\left(cX^p\right)\\
&=\sum_{i\in\mathbb{N}}\left(a\cdot\pi_i^m(b)\right)X^{i+n}\cdot \sum_{j\in\mathbb{N}} \alpha_{j+1,p+1}(c)X^j\\
&=\sum_{i\in\mathbb{N}}\sum_{j\in\mathbb{N}}\left(a\cdot\pi_i^m(b)\right)X^{i+n}\cdot \alpha_{j+1,p+1}(c)X^j\\
&=\sum_{i\in\mathbb{N}}\sum_{j\in\mathbb{N}}\sum_{l\in\mathbb{N}}\left(a\cdot\pi_i^m(b)\right)\cdot \pi_{l}^{i+n}\left(\alpha_{j+1,p+1}(c)\right)X^{l+j}\\
&=\sum_{k\in\mathbb{N}}\sum_{j\in\mathbb{N}}\sum_{i\in\mathbb{N}}\left(a\cdot\pi_i^m(b)\right)\cdot \pi_{k-j}^{i+n}\left(\alpha_{j+1,p+1}(c)\right)X^{k}.
\end{align*}}
Comparing coefficients completes the proof.
\end{proof}

\begin{corollary}\label{cor:hom-ore-ness} Let $R[X;\sigma,\delta]$ be a non-unital, hom-associative Ore extension of a non-unital, non-associative ring $R$, with twisting map defined by \eqref{eq:alpha-definition}. Then the following assertions hold for all $a,b,c\in R$ and $k,p\in\mathbb{N}$:
{\allowdisplaybreaks
\begin{align}
\sum_{i=k-p}^{I_{0,a}}\alpha_{i+1,1}(a)\cdot\pi_{k-p}^i(b\cdot c)=&(a\cdot b)\cdot\alpha_{k+1,p+1}(c),\label{eq:hom-ness-1}\\
\sum_{i=k-p-1}^{I_{0,a}}\alpha_{i+1,1}(a)\cdot\left(\pi_{k-p-1}^i(b\cdot\sigma(c))\right)\phantom{=}\nonumber\\
+\sum_{i=k-p}^{I_{0,a}}a_{i+1,1}\left(\pi_{k-p}^i(b\cdot\delta(c))\right)=&(a\cdot b)\cdot\left(\delta(\alpha_{k+1,p+1}(c))+\sigma(\alpha_{k,p+1}(c))\right)\nonumber\\
=&(a\cdot b)\cdot\left(\alpha_{k+1,p+1}(\delta(c))+\alpha_{k,p+1}(\sigma(c))\right)\hspace{-.5ex},\label{eq:hom-ness-2}\\
\sum_{i=k-p}^{I_{1,a}}\alpha_{i+1,2}(a)\cdot\pi^i_{k-p}(b\cdot c)=&\left(a\cdot\sigma(b)\right)\cdot\left(\delta(\alpha_{k+1,p+1}(c))+\sigma(\alpha_{k,p+1}(c))\right)\nonumber\\
&+(a\cdot\delta(b))\cdot\alpha_{k+1,p+1}(c),\label{eq:hom-ness-3}
\end{align}}
where $\alpha_{0,p+1}(\cdot):=0$, and $I_{p,a}$ is the smallest natural number, depending on $p$ and $a$, such that $\alpha_{i+1,p}(a)=0$ for all $i> I_{p,a}$.
\end{corollary}

\begin{proof} We get \eqref{eq:hom-ness-1}, the first equality in \eqref{eq:hom-ness-2}, and \eqref{eq:hom-ness-3} immediatly from the cases $m=n=0$, $m=0, n=1$, and $m=1,n=0$ in \eqref{eq:hom-ore-general}, respectively. The second equality in \eqref{eq:hom-ness-2} follows from comparison with \eqref{eq:hom-ness-1}.
\end{proof}

\begin{remark} In case $k<p$, or $k>I_{0,a}$ in \eqref{eq:hom-ness-1}, $(a\cdot b)\cdot\alpha_{k+1,p+1}(c)=0$. The statement is analogous for \eqref{eq:hom-ness-2} and \eqref{eq:hom-ness-3}.
\end{remark}

\begin{corollary}
Let $R[X;\sigma,\delta]$ be a non-unital, hom-associative Ore extension of a non-unital, non-associative ring $R$, with twisting map defined by \eqref{eq:alpha-definition}. Then the following assertions hold for all $a,b,c\in R$ and $j,p\in\mathbb{N}$:
\begin{align}
(a\cdot b)\cdot \sigma(\alpha_{I+1,p+1}(c))&=(a\cdot b)\cdot\alpha_{I+1,p+1}(\sigma(c)),\quad I=\max(I_{p,c}, I_{p,\delta(c)}),\label{eq:assertion1}\\
(a\cdot b)\cdot\delta(\alpha_{1,p+1}(c))&=(a\cdot b)\cdot \alpha_{1,p+1}(\delta(c))=\begin{cases}(a\cdot b)\cdot \alpha_{j+1,j+1}(\delta(c))&\text{if } p=0,\\ 0&\text{if }p\neq0.\label{eq:assertion2}\end{cases}
\end{align}
\end{corollary}

\begin{proof}
Put $k=\max(I_{p,c}, I_{p,\delta(c)})$ and $k=0$ in \eqref{eq:hom-ness-2}, respectively.
\end{proof}

\section{Hom-associative Ore extensions of hom-associative rings}
In this section, we will continue our previous investigation, but narrowed down to hom-associative Ore extensions of hom-associative rings.
\begin{corollary}\label{cor:hom-alpha-condition}Let $R[X;\sigma,\delta]$ be a non-unital, non-associative Ore extension of a non-unital, hom-associative ring $R$, and extend the twisting map $\alpha\colon R\to R$ homogeneously to $R[X;\sigma,\delta]$. Then $R[X;\sigma,\delta]$ is hom-associative if and only if for all $a,b,c\in R$ and $l,m,n\in\mathbb{N}$,
\begin{equation}
\sum_{i\in\mathbb{N}} \alpha(a)\cdot\pi_i^m\left(b\cdot\pi_{l-i}^n(c)\right) = \sum_{i\in\mathbb{N}}\left(a\cdot\pi_i^m(b)\right)\cdot\pi_l^{i+n}\left(\alpha(c)\right).\label{eq:pi-function-sum}
\end{equation}
\end{corollary}

\begin{proof}A homogeneous $\alpha$ corresponds to $\alpha_{i+1,m+1}(a)=\alpha(a)\cdot\delta_{i,m}$ and $\alpha_{j+1,p+1}(c)=\alpha(c)\cdot\delta_{j,p}$ in \autoref{prop:hom-ass-ore-condition}, where $\delta_{i,m}$ is the Kronecker delta. Then the left-hand side reads
\begin{align*}
\sum_{j\in\mathbb{N}}\sum_{i\in\mathbb{N}}\alpha_{i+1,m+1}(a)\cdot\pi_{k-j}^i\left(b\cdot\pi_{j-p}^n(c)\right)&=\sum_{j\in\mathbb{N}}\sum_{i\in\mathbb{N}}\alpha(a)\cdot\delta_{i,m}\cdot\pi_{k-j}^i\left(b\cdot\pi_{j-p}^n(c)\right)\\
&=\sum_{j\in\mathbb{N}}\alpha(a)\cdot\pi_{k-j}^m\left(b\cdot\pi_{j-p}^n(c)\right)\\
&=\sum_{i\in\mathbb{N}}\alpha(a)\cdot\pi_i^m\left(b\cdot\pi_{k-p-i}^n(c)\right)\\
&=\sum_{i\in\mathbb{N}}\alpha(a)\cdot\pi_i^m\left(b\cdot\pi_{l-i}^n(c)\right),
\end{align*}
and the right-hand side
\begin{align*}
\sum_{j\in\mathbb{N}}\sum_{i\in\mathbb{N}}\left(a\cdot\pi_i^m(b)\right)\cdot \pi_{k-j}^{i+n}\left(\alpha_{j+1,p+1}(c)\right)&=\sum_{j\in\mathbb{N}}\sum_{i\in\mathbb{N}}\left(a\cdot\pi_i^m(b)\right)\cdot \pi_{k-j}^{i+n}\left(\alpha(c)\cdot\delta_{j,p}\right)\\
&=\sum_{i\in\mathbb{N}}\left(a\cdot\pi_i^m(b)\right)\cdot \pi_{k-p}^{i+n}\left(\alpha(c)\right)\\
&=\sum_{i\in\mathbb{N}}\left(a\cdot\pi_i^m(b)\right)\cdot \pi_{l}^{i+n}\left(\alpha(c)\right),
\end{align*}
which completes the proof.
\end{proof}

\begin{corollary}Let $R[X;\sigma,\delta]$ be a non-unital, hom-associative Ore extension of a non-unital, hom-associative ring $R$, with the twisting map $\alpha\colon R\to R$ extended homogeneously to $R[X;\sigma,\delta]$. Then, for all $a,b,c\in R$,
\begin{align}
(a\cdot b)\cdot\delta(\alpha(c))&=(a\cdot b)\cdot\alpha(\delta(c)),\label{eq:homogeneous-delta-commute}\\
(a\cdot b)\cdot\sigma(\alpha(c))&=(a\cdot b)\cdot\alpha(\sigma(c)),\label{eq:homogeneous-sigma-commute}\\
\alpha(a)\cdot\delta(b\cdot c)&=\alpha(a)\cdot(\delta(b)\cdot c+\sigma(b)\cdot\delta(c)),\label{eq:homogeneous-sigma-derivation}\\
\alpha(a)\cdot\sigma(b\cdot c)&=\alpha(a)\cdot\left(\sigma(b)\cdot\sigma(c)\right).\label{eq:homogeneous-sigma-homomorphism}
\end{align}
\end{corollary}

\begin{proof}Using the same technique as in the proof of \autoref{cor:hom-alpha-condition}, this follows from \autoref{cor:hom-ore-ness} with a homogeneous $\alpha$.
\end{proof}

For the two last equations, it is worth noting the resemblance to the unital and associative case (see the latter part of \autoref{re:sigma-derivation}).

\begin{corollary}
Assume $\alpha\colon R\to R$ is the twisting map of a non-unital, hom-associative ring $R$, and extend the map homogeneously to $R[X;\sigma, \delta]$. Assume further that $\alpha$ commutes with $\delta$ and $\sigma$. Then $R[X;\sigma,\delta]$ is hom-associative if and only if for all $a,b,c\in R$ and $l,m,n\in \mathbb{N}$,
\begin{equation}
\alpha(a)\cdot\sum_{i\in\mathbb{N}}\pi_i^m\left(b\cdot\pi_{l-i}^n(c)\right)=\alpha(a)\cdot\sum_{i\in\mathbb{N}}\left(\pi_i^m(b)\cdot\pi_l^{i+n}(c)\right).\label{eq:homogeneous-iff-condition}
\end{equation}
\end{corollary}
\begin{proof}
Using \autoref{cor:hom-alpha-condition}, we know that $R[X;\sigma,\delta]$ is hom-associative if and only if for all $a,b,c\in R$ and $l,m,n\in\mathbb{N}$,
\begin{equation*}
\sum_{i\in\mathbb{N}}\alpha(a)\cdot\pi_i^m\left(b\cdot\pi_{l-i}^n(c)\right)=\sum_{i\in\mathbb{N}}\left(a\cdot\pi_i^m(b)\right)\cdot\pi_l^{i+n}\left(\alpha(c)\right).
\end{equation*}
However, since $\alpha$ commutes with both $\delta$ and $\sigma$, and $R$ is hom-associative, the right-hand side can be rewritten as
\begin{align*}
\sum_{i\in\mathbb{N}}\left(a\cdot\pi_i^m(b)\right)\cdot\pi_l^{i+n}\left(\alpha(c)\right)&=\sum_{i\in\mathbb{N}}\left(a\cdot\pi_i^m(b)\right)\cdot\alpha\left(\pi_l^{i+n}(c)\right)\\
&=\sum_{i\in\mathbb{N}}\alpha(a)\cdot\left(\pi_i^m(b)\cdot\pi_l^{i+n}(c)\right).
\end{align*}
As a last step, we use left-distributivity to pull out $\alpha(a)$ from the sums.
\end{proof}

\begin{proposition}\label{prop:extended-ore} Assume $\alpha\colon R\to R$ is the twisting map of a non-unital, hom-associative ring $R$, and extend the map homogeneously to $R[X;\sigma, \delta]$. Assume further that $\alpha$ commutes with $\delta$ and $\sigma$, and that $\sigma$ is an endomorphism and $\delta$ a $\sigma$-derivation. Then $R[X;\sigma,\delta]$ is hom-associative.
\end{proposition}

\begin{proof}We refer the reader to the proof in \cite{NYSTEDT20132748}, where it is seen that neither associativity, nor unitality is used to prove that for all $b,c\in R$ and $l,m,n\in\mathbb{N}$,
\begin{equation}
\sum_{i\in\mathbb{N}}\pi_i^m\left(b\cdot\pi_{l-i}^n(c)\right) = \sum_{i\in\mathbb{N}}\pi_i^m(b)\cdot\pi_l^{i+n}\left(c\right), \label{eq:assoc-id}
\end{equation}
and therefore also \eqref{eq:homogeneous-iff-condition} holds.
\end{proof}
One may further ask oneself whether it is possible to construct non-trivial hom-associative Ore extensions, starting from associative rings? The answer is affirmative, and the remaining part of this section will be devoted to show that.

\begin{proposition}\label{prop:hom*ore} Let $R[X;\sigma,\delta]$ be a non-unital, associative Ore extension of a non-unital, associative ring $R$, and $\alpha\colon R\to R$ a ring endomorphism that commutes with $\delta$ and $\sigma$. Then $\left(R[X;\sigma,\delta],*,\alpha\right)$ is a multiplicative, non-unital, hom-associative Ore extension with $\alpha$ extended homogeneously to $R[X;\sigma,\delta]$.
\end{proposition}

\begin{proof}
Since $\alpha$ is an endomorphism on $R$ that commutes with $\delta$ and $\sigma$, \eqref{eq:homogeneous-endomorphism} holds, so by \autoref{lem:endo-extension}, the homogeneously extended map $\alpha$ on $R[X;\sigma,\delta]$ is an endomorphism. Referring to \autoref{prop:star-alpha-mult}, $\left(R[X;\sigma,\delta],*,\alpha\right)$ is thus a hom-associative ring. Furthermore, we see that $*$ is the multiplication \eqref{eq:ore-mult} of a non-unital, non-associative Ore extension, since for all $a,b\in R$ and $m,n\in\mathbb{N}$,
\begin{align*}
aX^m*bX^n&=\alpha\left(\sum_{i\in\mathbb{N}}\left(a\cdot\pi_i^m(b)\right)X^{i+n}\right)=\sum_{i\in\mathbb{N}}\alpha\left(a\cdot\pi_i^m(b)\right)X^{i+n}\\
&=\sum_{i\in\mathbb{N}}\left(a*\pi_i^m(b)\right)X^{i+n}.
\end{align*}
\end{proof}
\begin{remark}\label{re:weak-unit}Note in particular that if $R[X;\sigma,\delta]$ is unital, then $(R[X;\sigma,\delta],*,\alpha)$ is weakly unital with weak unit 1 due to \autoref{cor:weak-unit}. \end{remark}
\begin{proposition}[Hom-associative $\sigma$-derivation]\label{{prop:hom-associative-sigma}} Let $A$ be an associative algebra, $\alpha$ and $\sigma$ algebra endomorphisms, and $\delta$ a $\sigma$-derivation on $A$. Assume $\alpha$ commutes with $\delta$ and $\sigma$. Then $\sigma$ is an algebra endomorphism and $\delta$ a $\sigma$-derivation on $(A,*,\alpha)$.
\end{proposition}

\begin{proof}
Linearity follows immediately, while for any $a,b\in A$,
\begin{align*}
\sigma(a* b)&=\sigma(\alpha(a\cdot b))=\alpha(\sigma(a\cdot b))=\alpha(\sigma(a)\cdot \sigma(b))=\sigma(a)* \sigma(b),\\
\delta(a* b)&=\delta(\alpha(a\cdot b))=\alpha(\delta(a\cdot b))=\alpha(\sigma(a)\cdot \delta(b)+\delta(a)\cdot b)\\
&=\alpha(\sigma(a)\cdot \delta(b))+\alpha(\delta(a)\cdot b)=\sigma(a)*\delta(b)+\delta(a)* b,
\end{align*}
which completes the proof.
\end{proof}

\begin{remark}\label{re:skew-ore} For a non-unital, associative skew polynomial ring $R[X;\sigma,0]$, one can always achieve a deformation into a non-unital, hom-associative skew polynomial ring using \autoref{prop:hom*ore} by defining the twisting map $\alpha$ as $\sigma$, due to the fact that $\sigma$ always commutes with itself and the zero map.
\end{remark}

\begin{example}[Hom-associative quantum planes]\label{ex:hom-quant} The quantum plane can be defined as the unital, associative skew polynomial ring $K[Y][X;\sigma,0]=:A$ where $K$ is a field of characteristic zero and $\sigma$ the unital $K$-algebra automorphism of $K[Y]$ such that $\sigma(Y)=qY$ and $q\in K^\times$, $K^\times$ being the multiplicative group of nonzero elements in $K$. From \autoref{re:skew-ore}, we know that at least one nontrivial deformation of $K[Y][X;\sigma,0]$ into a hom-associative skew polynomial ring exist, so let us try to see if there are others as well. Putting $\alpha(Y)=a_mY^m+\ldots+a_1Y + a_0$ for some constants $a_m,\ldots,a_0\in K$ and $m\in\mathbb{N}$ and then comparing $\sigma(\alpha(Y))=a_m q^m Y^m+\ldots+a_1 qY+a_0$ and $\alpha(\sigma(Y))=\alpha(qY)=q\alpha(Y)=a_m qY^m+\ldots+a_1qY+a_0q$ gives $\alpha(Y)=a_1Y$ since $q\in K^\times$ is arbitrary. By the same kind of argument, $\alpha(\sigma(1))=\sigma(\alpha(1))$ if and only if $\alpha(1)=1$. For such $\alpha$ and any monomial $b_n Y^n$ where $b_n\in K$ and $n\in\mathbb{N}_{>0}$,
\begin{align*}
\alpha(\sigma(b_n Y^n))=&b_n\alpha(\sigma(Y^n))=b_n\alpha(\sigma^n(Y))=b_n\alpha^n(\sigma(Y))=b_n\sigma^n(\alpha(Y))\\
=&b_n\sigma(\alpha^n(Y))=b_n\sigma(\alpha(Y^n))=\sigma(\alpha(b_nY^n)).
\end{align*}
By linearity, $\alpha$ commutes with $\sigma$ on any polynomial in $K[Y][X;\sigma,0]$, and by excluding the possibility $\alpha\equiv0$, we put the twisting map to be $\alpha_k(Y)=kY$ for some element $k\in K^\times$, the index $k$ making evident that the map depends on the parameter $k$. This $K$-algebra endomorphism gives us a family of hom-associative quantum planes $(A,*,\alpha_k)$, each value of $k$ giving a weakly unital hom-associative skew polynomial ring, the member for which $k=1$ corresponding to the unital, associative quantum plane. If $k\neq1$, we get nontrivial deformations, since for instance $X*(Y*Y)=k^4q^2 Y^2X$, while $(X*Y)*Y=k^3q^2Y^2X$. Now that \autoref{{prop:hom-associative-sigma}} guarantees that $\sigma$ is a $K$-algebra endomorphism on any member of $(A,*,\alpha_k)$ as well, we call these members \emph{hom-associative quantum planes}, satisfying the commutation relation $X*Y=kqY*X$.
\end{example}

\begin{example}[Hom-associative universal enveloping algebras]\label{ex:hom-env} The two-di\-men\-sio\-nal Lie algebra $L$ with basis $\{X,Y\}$ over the field $K$ of characteristic zero is defined by the Lie bracket $[X,Y]_L=Y$. Its universal enveloping algebra, $U(L)$, can be written as the unital, associative differential polynomial ring $K[Y][X;\mathrm{id}_{K[Y]},\delta]$ where $\delta=Y\frac{\mathrm{d}}{\mathrm{d}Y}$. Put for the $K$-algebra endomorphism $\alpha(Y)=a_n Y^n+\ldots+ a_1 Y+ a_0$ where $a_n,\ldots,a_0\in K$ and $n\in\mathbb{N}$. Then $\alpha(\delta(Y))=\alpha(Y)=a_n Y^n+\ldots+ a_1 Y+ a_0$ and $\delta(\alpha(Y))=n a_n Y^n+\ldots+a_1 Y$, so by comparing coefficients, $a_1$ is the only nonzero such. Using the same kind of argument, $\alpha(\delta(1))=\delta(\alpha(1))$ if and only if $\alpha(1)=1$. Let $b_n\in K$ be arbitrary and $m\in\mathbb{N}_{>0}$. Then $\alpha(\delta(b_n Y^n))=nb_n \alpha(Y^n)=nb_n\alpha^n(Y)=nb_na_1^n Y^n$, and $\delta(\alpha(b_nY^n))=\delta(b_n\alpha^n(Y))=\delta(b_n a_1^nY^n)=nb_n a_1^n Y^n$. Since it is sufficient to check commutativity of $\alpha$ and $\delta$ on an arbitrary monomial, we define the twisting map as $\alpha_k(Y)=kY, k\in K^\times$, giving a family of \emph{hom-associative universal enveloping algebras of $L$}, $(U(L), *, \alpha_k)$, where the commutation relation $X\cdot Y-Y\cdot X=Y$ is deformed to $X*Y-Y*X=kY$.
\end{example}

\begin{example}[Hom-associative Weyl algebras]\label{ex:hom-weyl}
Consider the first Weyl algebra exhibited as a unital, associative differential polynomial ring, $K[Y][X;\mathrm{id}_{K[Y]},\delta]=:A$, where $K$ is a field of characteristic zero and $\delta=\frac{\mathrm{d}}{\mathrm{d}Y}$. Clearly any algebra endomorphism $\alpha$ on $K[Y]$ commutes with $\mathrm{id}_{K[Y]}$, but what about $\delta$? Since $\alpha(\delta(Y))=\alpha(1)=1$, we need to have $\delta(\alpha(Y))=1$ which implies $\alpha(Y)=Y+k$, for some $k\in K$. On the other hand, if $\alpha$ is an algebra endomorphism such that $\alpha(Y)=Y+k$ for any $k\in K$, then for any monomial $aY^m$ where $m\in\mathbb{N}_{>0}$,
\begin{align*}
\alpha(\delta(aY^m))&=am\alpha(Y^{m-1})=am\alpha^{m-1}(Y)=a m(Y+k)^{m-1},\\
\delta(\alpha(aY^m))&=a\delta(\alpha^m(Y))=a\delta((Y+k)^m)=a m(Y+k)^{m-1}.
\end{align*}
Hence any algebra endomorphism $\alpha$ on $K[Y]$ that satisfies $\alpha(Y)=Y+k$ for any $k\in K$ will commute with $\delta$ (and any algebra endomorphism that commutes with $\delta$ will be on this form). Since $\alpha$ commutes with $\delta$ and $\sigma$, we know from \autoref{cor:unital-endo-extension} that $\alpha$ extends to a ring endomorphism on $A$ as well by $\alpha(aX^m)=\alpha(a)X^m$. Linearity over $K$ follows from the definition, so in fact $\alpha$ extends to an algebra endomorphism on $A$. Appealing to \autoref{prop:hom*ore} and \autoref{re:weak-unit}, we thus have a family of hom-associative, weakly-unital differential polynomial rings $(A,*,\alpha_k)$ with weak unit 1, where $k\in K$ and $\alpha_k$ is the $K$-algebra endomorphism defined by $\alpha_k\left(p(Y)X^m\right)=p(Y+k)X^m$ for all polynomials $p(Y)\in K[Y]$ and $m\in\mathbb{N}$. Since \autoref{{prop:hom-associative-sigma}} assures $\delta$ to be a $K$-linear $\sigma$-derivation on any member $(A,*,\alpha_k)$ as well, we call these \emph{hom-associative Weyl algebras}, including the associative Weyl algebra in the member corresponding to $k=0$. One can note that the hom-associative Weyl algebras all satisfy the commutation relation $X*Y-Y*X=1$, where 1 is a weak unit.
\end{example}

\begin{lemma}\label{lem:diff-mult}
Let $R$ be a non-unital, non-associative ring. Then in $R[X;\mathrm{id}_R,\delta]$,
\begin{equation}
aX^n\cdot b=\sum_{i=0}^n \left(\binom{n}{i}\cdot a\cdot\delta^{n-i}(b)\right)X^i, \quad\text{ for any } a,b\in R \text{ and } n\in\mathbb{N}.
\end{equation}
\end{lemma}

\begin{proof}This follows from \eqref{eq:ore-mult} with $\sigma=\mathrm{id}_R$.
\end{proof}

\begin{lemma}\label{le:weak-unit-hom-conditions} Let $R$ be a weakly unital, hom-associative ring with weak unit $e$ and twisting map $\alpha$ commuting with the derivation $\delta$ on $R$, and extend $\alpha$ homogeneously to  $R[X;\mathrm{id}_R,\delta]$. Then the following hold:
\begin{enumerate}[(i)]
	\item $a\cdot\delta^n(e)=\delta^n(e)\cdot a=0$ for any $a\in R$ and $n\in\mathbb{N}_{>0}$,
	\item $e$ is a weak unit in $R[X;\mathrm{id}_R,\delta]$,
	\item $eX\cdot q -q\cdot eX =\sum_{i=0}^n \alpha(\delta(q_i))X^i$ for any $q=\sum_{i=0}^n q_i X^i\in R[X;\mathrm{id}_R,\delta]$.
\end{enumerate}
\end{lemma}

\begin{proof} First, note that
\begin{align*}
\delta(a\cdot e)=&a\cdot\delta(e)+\delta(a)\cdot e=a\cdot\delta(e)+e\cdot\delta(a),\\
\delta(a\cdot e)=&\delta(e\cdot a)=e\cdot\delta(a)+\delta(e)\cdot a,
\end{align*}
and hence $\delta(e)\cdot a=a\cdot\delta(e)$. Moreover, $\delta(a\cdot e)=\delta(\alpha(a))=\alpha(\delta(a))=e\cdot\delta(a)$, so $\delta(e)\cdot a=0$. Assume $\delta^n(e)\cdot a=a\cdot\delta^n(e)=0$ for all $n\in\mathbb{N}_{>0}$. Then, since $a$ is arbitrary, $\delta^n(e)\cdot\delta(a)=\delta(a)\cdot\delta^n(e)=0$ as well, and hence
\begin{align*}
0=&\delta(0)=\delta\left(a\cdot\delta^n(e)\right)=a\cdot\delta^{n+1}(e)+\delta(a)\cdot\delta^n(e)=a\cdot\delta^{n+1}(e),\\
0=&\delta(0)=\delta\left(\delta^n(e)\cdot a\right)=\delta^n(e)\cdot\delta(a)+\delta^{n+1}(e)\cdot a= \delta^{n+1}(e)\cdot a,
\end{align*}
so the first assertion holds by induction. The second assertion follows from the first and \autoref{lem:diff-mult} with $b=e$, since for any $m\in\mathbb{N}$,
\begin{equation*}
aX^m\cdot e=(a\cdot e)X^m=\alpha(a)X^m=\alpha\left(aX^m\right)=(e\cdot a)X^m=e\cdot \left(a X^m\right),
\end{equation*}
and by distributivity of the multiplication, $e\cdot q=q\cdot e=\alpha(q)$ for any $q\in R[X;\mathrm{id}_R;\delta]$. The last assertion follows from a direct computation using the first assertion and \autoref{lem:diff-mult}.\qedhere
\end{proof}

A well-known fact about the associative Weyl algebras are that they are simple. This fact is also true in the case of the non-associative Weyl algebras introduced in \cite{2015arXiv150901436N}, and it turns out that the hom-associative Weyl algebras have this property as well.

\begin{proposition}The hom-associative Weyl algebras are simple.
\end{proposition}

\begin{proof}
The main part of the proof follows the same line of reasoning that can be applied to the unital and associative case; let $(A,*,\alpha_k)$ be any hom-associative Weyl algebra, and $I$ any nonzero ideal of it. Let $p=\sum_{i\in\mathbb{N}}p_i(Y)X^i\in I$ be an arbitrary nonzero polynomial with $p_i(Y)\in K[Y]$, and put $m:=\max_i(\deg(p_i(Y)))$. Then, since $1\in A$ is a weak unit in $(A,*,\alpha_k)$, we may use \autoref{le:weak-unit-hom-conditions} and the commutator $[\cdot,\cdot]$ to compute
\begin{equation*}
[X,p]=\sum_{i\in\mathbb{N}}\alpha_k\left(p'_i(Y)X^i\right)=\sum_{i\in\mathbb{N}}p'_i(Y+k)X^i.
\end{equation*}
Since $\max_i(\deg(p'_i(Y+k))=m-1$, by applying the commutator to the resulting polynomial with $X$ $m$ times, we get a polynomial $\sum_{j\in\mathbb{N}}a_jX^j$ of degree $n$, where $a_n\in K$ is nonzero. Then
\begin{align*}
\sum_{j\in\mathbb{N}}a_jX^j*Y&=\sum_{j\in\mathbb{N}}\sum_{i\in\mathbb{N}}a_j*\pi_i^j(Y)X^i=\sum_{j\in\mathbb{N}}\left(a_j*X^{j-1}+a_j*YX^j\right),\\
Y*\sum_{j\in\mathbb{N}}a_jX^j&=\alpha_k\left(Y\sum_{j\in\mathbb{N}}a_jX^j\right)=\alpha_k\left(\sum_{j\in\mathbb{N}}a_jYX^j\right)=\sum_{j\in\mathbb{N}}\alpha_k\left(a_jYX^j\right)\\
&=\sum_{j\in\mathbb{N}}a_j*YX^j.
\end{align*}
Therefore $\deg\left(\left[\sum_{j\in\mathbb{N}}a_jX^j,Y\right]\right)=n-1$, where $\deg(\cdot)$ now denotes the degree of a polynomial in $X$. By applying the commutator to the resulting polynomial with $Y$ $n$ times, we get $a_n*1\in I$;
\begin{equation*}
a_n*1=\alpha_k(a_n)=a_n\in I\implies a_n^{-1}*\left(a_n*1\right)=a_n^{-1}*a_n=\alpha_k(1)=1\in I.
\end{equation*}
Take any polynomial $q=\sum_{i\in\mathbb{N}}q_i(Y)X^i$ in $(A,*,\alpha_k)$. Then
\begin{equation*}
1*\sum_{i\in\mathbb{N}}q_i(Y-k)X^i=\sum_{i\in\mathbb{N}}q_i(Y)X^i=q\in I,\text{ and therefore } I=(A,*,\alpha_k).
\end{equation*}
\end{proof}

\section{Weak unitalizations of hom-associative algebras}\label{sec:weak-unitalization}
For a non-unital, associative $R$-algebra $A$ consisting of an $R$-module $M$ endowed with a multiplication, one can always find an embedding of the algebra into a unital, associative algebra by taking the direct sum $M\oplus R$ and defining multiplication by
\begin{equation*}
(m_1,r_1)\cdot(m_2,r_2):=(m_1\cdot m_2+r_1\cdot m_2+r_2\cdot m_1, r_1\cdot r_2),\quad m_1,m_2\in M \text{ and } r_1,r_2\in R.
\end{equation*}
$A$ can then be embedded by the injection map $M\to M\oplus 0$, being an isomorphism into the unital, associative algebra $M\oplus R$ with the unit given by $(0,1)$.

In \cite{2009arXiv0904.4874F}, Frégier and Gohr showed that not all hom-associative algebras can be embedded into even a weakly unital hom-associative algebra. In this section, we prove that any multiplicative hom-associative algebra can be embedded into a multiplicative, weakly unital hom-associative algebra by twisting the above unitalization of a non-unital, associative algebra with $\alpha$. We call this a \emph{weak unitalization}.

\begin{proposition}\label{prop:bullet-algebra} Let $M$ be a non-unital, non-associative $R$-algebra and $\alpha$ a linear map on $M$. Endow $M\oplus R$ with the following multiplication:
\begin{align}
(m_1,r_1)\bullet(m_2,r_2):=&(m_1\cdot m_2+r_1\cdot \alpha(m_2)+r_2\cdot \alpha(m_1),r_1\cdot r_2),\label{eq:bullet-mult}
\end{align}
for any $m_1,m_2\in M$ and $r_1,r_2\in R$. Then $M\oplus R$ is a non-unital, non-associative $R$-algebra.
\end{proposition}

\begin{proof}
$R$ can be seen as a module over itself, and since any direct sum of modules over $R$ is again a module over $R$, $M\oplus R$ is a module over $R$. For any $m_1,m_2\in M$ and $\lambda, r_1, r_2\in R$,
\begin{align*}
\lambda\cdot\left((m_1,r_1)\bullet(m_2,r_2)\right)&=\lambda\cdot\left(m_1\cdot m_2+r_1\cdot \alpha(m_2)+r_2\cdot \alpha(m_1),r_1\cdot r_2\right)\\
&=\left(\lambda\cdot m_1\cdot m_2+\lambda\cdot r_1\cdot\alpha(m_2)+\lambda\cdot r_2\cdot\alpha(m_1),\lambda\cdot r_1\cdot r_2\right)\\
&=\left(\lambda\cdot m_1\cdot m_2+\lambda\cdot r_1\cdot\alpha(m_2)+ r_2\cdot\alpha(\lambda\cdot m_1),\lambda\cdot r_1\cdot r_2\right)\\
&=(\lambda\cdot m_1,\lambda\cdot r_1)\bullet(m_2,r_2)=\left(\lambda\cdot(m_1,r_1)\right)\bullet(m_2,r_2),
\end{align*}
\begin{align*}
\left((m_1,r_1)+(m_2,r_2)\right)\bullet(m_3,r_3)=&(m_1+m_2,r_1+r_2)\bullet(m_3,r_3)\\
=&\left((m_1+m_2)\cdot m_3+(r_1+r_2)\cdot\alpha(m_3)\right.\\
&\left.+ r_3\cdot\alpha(m_1+m_2),(r_1+r_2)\cdot r_3)\right)\\
=&\left(m_1\cdot m_3+r_1\cdot\alpha(m_3)+r_3\cdot\alpha(m_1),r_1\cdot r_3\right)\\
&+\left(m_2\cdot m_3+m_2\cdot\alpha(m_3)+r_3\cdot\alpha(m_2),r_2\cdot r_3\right)\\
=&(m_1,r_1)\bullet(m_3,r_3)+(m_2,r_2)\bullet(m_3,r_3),
\end{align*}
so the binary operation $\bullet$ is linear in the first argument, and by symmetry, also linear in the second argument.
\end{proof}

\begin{proposition}[Weak unitalization] If $(M,\cdot,\alpha)$ is a multiplicative hom-associ\-ative algebra over an associative, commutative, and unital ring $R$, then $(M\oplus R,\bullet ,\beta_\alpha)$ is a multiplicative, weakly unital hom-associative algebra over $R$ with weak unit $(0,1)$. Here, $\bullet$ is given by \eqref{eq:bullet-mult} and $\beta_\alpha\colon M\oplus R\to M\oplus R$ by
\begin{align}
 \beta_\alpha((m_1,r_1)):=&(\alpha(m_1),r_1),\quad \text{for any } m_1\in M \text{ and } r_1\in R.
 \end{align}
We call $(M\oplus R,\bullet, \beta_\alpha)$ a \emph{weak unitalization} of $(M,\cdot,\alpha)$.
\end{proposition}
\begin{proof}
We proved in \autoref{prop:bullet-algebra} that the multiplication $\bullet$ made $M\oplus R$ a non-unital, non-associative algebra, and due to the fact that $\alpha$ is linear, it follows that $\beta_\alpha$ is also linear. Multiplicativity of $\beta_\alpha$ also follows from that of $\alpha$, since for any $m_1,m_2,m_3\in M$ and $r_1,r_2,r_3\in R$,
\begin{align*}
\beta_\alpha\left((m_1,r_1)\right)\bullet\beta_\alpha\left((m_2,r_2)\right)=&(\alpha(m_1),r_1)\bullet(\alpha(m_2),r_2)\\
=&\left(\alpha(m_1)\cdot\alpha(m_2)+r_1\cdot\alpha(\alpha(r_2))+r_2\cdot\alpha(\alpha(r_1)), r_1\cdot r_2\right)\\
=&\left(\alpha(m_1\cdot m_2)+r_1\cdot\alpha(\alpha(r_2))+r_2\cdot\alpha(\alpha(r_1)),r_1\cdot r_2\right)\\
=&\left(\alpha\left(m_1\cdot m_2+r_1\cdot\alpha(r_2)+r_2\cdot\alpha(r_1)\right),r_1\cdot r_2\right)\\
=&\beta_\alpha\left((m_1,r_1)\bullet(m_2,r_2)\right),
\end{align*}
while hom-associativity can be proved by the following calculation:
\begin{align*}
\beta_\alpha\left((m_1,r_1)\right)\bullet\left((m_2,r_2)\bullet(m_3,r_3)\right)=&(\alpha(m_1),r_1)\bullet(m_2\cdot m_3+r_2\cdot\alpha(m_3)\\
&+r_3\cdot\alpha(m_2),r_2\cdot r_3)\\
=&\left(\alpha(m_1)\cdot(m_2\cdot m_3)+r_2\cdot\alpha(m_1)\cdot\alpha(m_3)\right.\\
&+r_3\cdot\alpha(m_1)\cdot\alpha(m_2)\\
&+r_1\cdot\alpha(m_2\cdot m_3+ r_2\cdot\alpha(m_3)+r_3\cdot\alpha(m_2))\\
&\left.+r_2\cdot r_3\cdot\alpha(\alpha(m_1)), r_1\cdot r_2\cdot r_3\right)\\
=&\left((m_1\cdot m_2)\cdot\alpha(m_3)+r_2\cdot\alpha(m_1)\cdot\alpha(m_3)\right.\\
&+r_3\cdot\alpha(m_1)\cdot\alpha(m_2)\\
&+r_1\cdot\alpha(m_2\cdot m_3+ r_2\cdot\alpha(m_3)+r_3\cdot\alpha(m_2))\\
&\left.+r_2\cdot r_3\cdot\alpha(\alpha(m_1)), r_1\cdot r_2\cdot r_3\right)\\
=&\left((m_1\cdot m_2+r_1\cdot\alpha(m_2)+r_2\cdot\alpha(m_1))\cdot\alpha(m_3)\right.\\
&+r_1\cdot r_2\cdot\alpha(\alpha(m_3))+r_3\cdot\alpha(m_1\cdot m_2)\\
&+\left. r_3\cdot\alpha(r_1\cdot\alpha(m_2)+r_2\cdot\alpha(m_1)),r_1\cdot r_2\cdot r_3\right)\\
=&\left((m_1,r_1)\bullet (m_2,r_2)\right)\bullet \beta_\alpha((m_3,r_3)).
\end{align*}
At last, $(m_1,r_1)\bullet(0,1)=(0,1)\bullet(m_1,r_1)=(1\cdot\alpha(m_1),1\cdot r_1)=\beta_\alpha((m_1,r_1))$.
\end{proof}

\begin{remark} In case $\alpha$ is the identity map, so that the algebra is associative, the weak unitalization is the unitalization described in the beginning of this section, thus giving a unital algebra.
\end{remark}

\begin{corollary}\label{cor:iso-bullet} $(M,\cdot,\alpha)\cong(M\oplus 0,\bullet,\beta_\alpha)$.
\end{corollary}

\begin{proof} The projection map $\pi\colon M\oplus 0\to M$ is a bijective algebra homomorphism. For any $m\in M$, $\pi(\beta_\alpha(m,0))=\pi(\alpha(m),0)=\alpha(m)$ and $\alpha(\pi(m,0))=\alpha(m)$, therefore $\alpha\circ\pi=\pi\circ\beta_\alpha$, so by \autoref{def:morphism}, $(M\oplus 0,\bullet,\beta_\alpha)\cong (M,\cdot,\alpha)$.
\end{proof}

Using \autoref{cor:iso-bullet}, we identify $(M,\cdot,\alpha)$ with its image in $(M\oplus R,\bullet,\beta_\alpha)$, seeing the former as embedded in the latter.

\begin{lemma}\label{lem:weakly-unital-hom-ideal}All ideals in a weakly unital hom-associative algebra are hom-ideals.
\end{lemma}

\begin{proof} Let $I$ be an ideal, $a\in I$ and $e$ a weak unit in a hom-associative algebra. Then $\alpha(a)=e\cdot a\in I$, so $\alpha(I)\subseteq I$.
\end{proof}

A simple hom-associative algebra is always hom-simple, the hom-associative Weyl algebras in \autoref{ex:hom-weyl} being examples thereof. The converse is also true if the algebra has a weak unit, due to \autoref{lem:weakly-unital-hom-ideal}.

\begin{corollary} $(M,\cdot,\alpha)$ is a hom-ideal in $(M\oplus R,\bullet,\beta_\alpha)$.
\end{corollary}

\begin{proof} For any $m_1,m_2$ and $r_1\in R$, $(m_1,r_1)\bullet (m_2,0)=(m_1\cdot m_2+r_1\cdot\alpha(m_2),0)\in M$, and $(m_2,0)\bullet (m_1,r_1)=(m_2\cdot m_1+r_1\cdot\alpha(m_2),0)\in M$, so $(M,\cdot,\alpha)$ is an ideal in a weakly unital hom-associative algebra, and by \autoref{lem:weakly-unital-hom-ideal} therefore also a hom-ideal.
\end{proof}

Recall that for a ring $R$, if there is a positive integer $n$ such that $n\cdot a=0$ for all $a\in R$, then the smallest such $n$ is the \emph{characteristic of the ring $R$}, $\characteristic(R)$. If no such positive integer exists, then one defines $\characteristic(R)=0$.

\begin{proposition}Let $R$ be a weakly unital hom-associative ring with weak unit $e$ and injective or surjective twisting map $\alpha$. If $n\cdot e\neq0$ for all $n\in\mathbb{Z}_{>0}$, then $\characteristic(R)=0$. If $n\cdot e=0$ for some $n\in\mathbb{Z}_{>0}$, then the smallest such $n$ is the characteristic of $R$.
\end{proposition}

\begin{proof}If $n\cdot e\neq0$ for all $n\in\mathbb{Z}_{>0}$, then clearly we cannot have $n\cdot a=0$ for all $a\in R$, and hence $\characteristic(R)=0$. Now assume $n$ is a positive integer such that $n\cdot e=0$. If $\alpha$ is injective, then for all $a\in R$,
\begin{equation*}
\alpha(n\cdot a)=n\cdot\alpha(a)=n\cdot (e\cdot a)=(n\cdot e)\cdot a=0\cdot a=0\iff n\cdot a=0.
\end{equation*}
On the other hand, if $\alpha$ is surjective, then for all $a\in R$, $a=\alpha(b)$ for some $b\in R$, and hence $n\cdot a=n\cdot\alpha(b)=n\cdot(e\cdot b)=(n\cdot e)\cdot b=0\cdot b=0.$
\end{proof}

\begin{proposition} Let $R:=(M,\cdot,\alpha)$ be a hom-associative ring, and define
\begin{equation*}
S:=\begin{cases}(M\oplus\mathbb{Z},\bullet,\beta_\alpha),&\text{if } \characteristic(R)=0,\\
(M\oplus\mathbb{Z}_n,\bullet,\beta_\alpha),&\text{if }\characteristic(R)=n.
 \end{cases}
\end{equation*}
Then the weak unitalization $S$ of $R$ has the same characteristic as $R$.
\end{proposition}

\begin{proof}This follows immediately by using the definition of the characteristic.
\end{proof}

The main conclusion to draw from this section is that any multiplicative hom-associative algebra can be seen as a multiplicative, weakly unital hom-associative algebra by its weak unitalization. The converse, that any weakly unital hom-associative algebra is necessarily multiplicative if also $\alpha(e)=e$, where $e$ is a weak unit, should be known. However, since we have not been able to find this statement elsewhere, we provide a short proof of it here for the convenience of the reader.

\begin{proposition}If $e$ is a weak unit in a weakly unital hom-associative algebra $A$, and $\alpha(e)=e$, then $A$ is multiplicative.
\end{proposition}

\begin{proof} For any $a,b\in A$, $\alpha(e)\cdot\left(a\cdot b\right)=e\cdot\left(a\cdot b\right)=\alpha\left(a\cdot b\right)$. Using hom-associativity, $\alpha(e)\cdot\left(a\cdot b\right)=\left(e\cdot a\right)\cdot \alpha(b)=\alpha(a)\cdot \alpha(b)$.
\end{proof}

\noindent {\bf Acknowledgment.} We would like to thank Lars Hellstr\"om for discussions leading to some of the results presented in the article.

\end{document}